\documentclass[a4paper]{amsart}
\usepackage[utf8]{inputenc}

\usepackage{amsmath}
\usepackage{amsthm}
\usepackage{amssymb}

\theoremstyle{plain}
\newtheorem{theorem}{Theorem}[section]
\newtheorem{lemma}[theorem]{Lemma}
\newtheorem{claim}[theorem]{Claim}

\newtheorem{corollary}[theorem]{Corollary}
\newtheorem*{theorem*}{Theorem}
\newtheorem*{proposition*}{Proposition}

\theoremstyle{definition}
\newtheorem{definition}[theorem]{Definition}
\newtheorem*{definition*}{Definition}

\theoremstyle{remark}

\theoremstyle{plain}

\newcommand\vf{\varphi}
\DeclareMathOperator{\tp}{tp}

\newcommand\restrict{\upharpoonright}
\newcommand\scht{:}

\usepackage{amssymb}
\usepackage{upref}
\usepackage{hyperref}
\hypersetup{
colorlinks,%
citecolor=black,%
filecolor=black,%
urlcolor=black,%
linkcolor=black,%
}
\usepackage{lastpage}
\newtheorem{theorema}{Theorem}

\theoremstyle{definition}

\theoremstyle{remark}
\newtheorem{example}[theorem]{Example}
\newtheorem{question}[theorem]{Question}
\newtheorem{condition}[theorem]{Condition}
\newtheorem{convention}[theorem]{Convention}

\DeclareMathOperator{\dcl}{dcl}
\DeclareMathOperator\icl{icl}
\DeclareMathOperator\cl{cl}

\DeclareMathOperator\dom{dom}

\newcommand\proves{\vdash}

\newcommand\oF{{\tilde F}}

\newcommand\RR{\mathbb R}
\newcommand\QQ{\mathbb Q}

\begin{document}

\title[Definable functions continuous on curves in o-minimal structures]{
Definable functions continuous on curves in o-minimal structures}

\author{Janak Ramakrishnan}

\address{Centro de Matemática e Aplicações Fundamentais\\
Av. Prof. Gama Pinto, 2\\
1649-003 Lisboa\\
Portugal
}

\email{janak@janak.org}

\date{\today}


\thanks{Thanks to Thomas Scanlon, my Ph.D. advisor, for discussing this problem
  with me, and for all his valuable help.  The central question of this paper
  was brought to my attention by Patrick Speissegger.  I would like to thank him
  and Leo Harrington for reading and commenting on various versions of the
  results.}

\begin{abstract}
We give necessary and sufficient conditions on a non-oscillatory curve in an o-minimal field such that, for any bounded definable function, the germ of the function on an initial segment of the curve can be continuously extended to a closed definable set.  This situation is translated into a question about types: What are the conditions on an $n$-type such that, for any bounded definable function, there is a definable set containing the type on which the function is continuous, and can be extended continuously to the set's closure?  All such types are definable, and we give the precise conditions that are equivalent to existence of a desired definable set.
\end{abstract}

\maketitle
\section{Introduction}

The study of o-minimal structures often encounters functions that are not
first-order definable in such structures.  These functions may be definable in
an o-minimal expansion of the original structure or lie in a Hardy field
extension of the field of germs of definable functions.  In this article, we
examine non-oscillatory curves in an o-minimal structure -- curves that may not be definable in the structure, but are ``well-behaved,'' in that their component functions do not oscillate with respect to the definable functions.  For example, $\langle t,e^t\rangle$ is non-oscillatory in $(\RR,+,\cdot,<,0,1)$, despite not being a definable curve, since $e^t$ does not oscillate with respect to any rational function.

We will answer a question that arose from an attempt to generalize Theorem 7.1 of \cite{Malgrange74}, on the existence of a formal solution to a differential
equation implying the existence of a $C^\infty$ solution with Taylor series the
formal solution.

\begin{question}\label{mainquestion}
Let $\bar\gamma$ be a given non-oscillatory curve.  For every bounded definable function $F$, is there a definable set $C$ containing an initial segment of $\bar\gamma$ such that $F\restrict C$ is continuous and extends continuously to $\cl(C)$?
\end{question}

The answer is not always ``yes,'' as shown by the curve $\langle t,-1/\ln t\rangle$ near $\bar 0$ and the function $\min(1,y/x)$ in the structure $(\RR,+,\cdot,<,0,1)$ (see Corollary \ref{curvecounterex}).

To answer this question, we use an elementary observation -- that any non-oscillatory curve in $n$ dimensions has associated to it a complete $n$-type -- to turn the question into one about types, namely: when is a type contained in a definable set on which the function is continuous and extends continuously to the closure?

The way that such a set containing a type can fail to exist is that, in some sense, the type lies in a ``gap'' between two regions on which the function takes very different values, but which share a common boundary point that is the limit of the type, in the sense of \cite{HrLo11}.

The right way to formalize these notions of ``gap'' and ``region'' comes from \cite{MaSt94}'s concept of ``scale'' (our terminology).  A type being out of scale with respect to a structure means that the image of this structure under any definable function is not cofinal or coinitial around the type in the base set of the type.  However, this concept turns out to be insufficient in the absence of certain guarantees on the order of the variables of the type.  To that end, this article introduces the notion of a ``decreasing type,'' which simplifies the use of scale.  Given a decreasing $n$-type over a set $A$ and realization of this type, $\bar c$, for each $i\le n$ we have an index $Q(i)\le i$ which gives the greatest coordinate $k\le i$ such that $\tp(c_i/A\bar c_{<k})$ is definable.

We restrict our discussion for the most part to finite types. A type is
\emph{finite} if its realizations are contained in a definable set. Corollary \ref{infinitetype} gives the result in the general case. In the case of original interest, where the objective was to extend a function to the endpoint of a non-oscillatory curve, the type is easily seen to be finite. We are then equipped to state our main theorem:

\begin{theorema}\label{maintheorem}
Let $M$ be an o-minimal field, and let $A\subseteq M$.  Let $p$ be a finite decreasing $n$-type over $A$.  Let $\bar c=\langle c_1,\ldots,c_n\rangle\models p$.  Then the following statements are equivalent:
\begin{enumerate}
\item\label{firstmain} For every $A$-definable bounded $n$-ary function, $F$, defined on $\bar c$, there is an $A$-definable set $C$ with $\bar c\in C$, such that $F\restrict C$ is continuous and extends continuously to $\cl(C)$.
\item\label{secondmain} There is $i_0\le n$ such that $\tp(c_i/A\bar c_{<i})$ is algebraic, definable, or out of scale on $A\bar c_{<Q(i)}$ for $i=i_0,\ldots,n$, and for $i<i_0$, $\tp(c_i/A)$ is not definable.
\end{enumerate}
\end{theorema}

Using our correspondence between types and curves (Lemma \ref{questionequiv}), we obtain our desired result:

\begin{theorema}\label{curvetheorem}
Let $M$ be an o-minimal field, and let $\bar\gamma(t)=\langle\gamma_1(t),\ldots,\gamma_n(t)\rangle$ be a (not necessarily definable) non-oscillatory bounded curve in $M^n$.  Then the following statements are equivalent:
\begin{enumerate}
\item\label{curveitem2} For every $A$-definable bounded $n$-ary function, $F$, defined on an initial segment of $\bar\gamma$, there is an $A$-definable set $C$ containing an initial segment of $\bar\gamma$, such that $F\restrict C$ is continuous and extends continuously to $\cl(C)$.
\item\label{*} The coordinates of $\bar\gamma$ can be reordered so that $\tp(\bar\gamma/M)$ is decreasing and there is $i_0\le n$ such that for any $\bar c\models\tp(\bar\gamma/M)$, the type $\tp(c_i/A\bar c_{<i})$ is algebraic, definable, or out of scale on $A\bar c_{<Q(i)}$ for $i=i_0,\ldots,n$, and for $i<i_0$, $\tp(c_i/A)$ is not definable.
\end{enumerate}
\end{theorema}

We note here that types satisfying \eqref{firstmain} of Theorem \ref{maintheorem} can be characterized using the framework of $T$-convexity developed in \cite{vdDLew95}. Statement \eqref{firstmain} for a type $p$ is equivalent to the following: if $\bar c\models p$, then the convex hull of $M$ in the prime model containing $M\bar c$ is given by elements of the form $F(\bar c)$, where $F$ is an $M$-definable, global continuous bounded $n$-ary function.\footnote{Thanks to an anonymous referee for pointing out this equivalence.} This equivalence is easily seen since a continuous function on a closed set can be definably extended to a global continuous function. Note that the convex hull of $M$ is a $T$-convex subring of the prime model of $M\bar c$. Thus, Statement \eqref{secondmain}, or, more precisely, \eqref{coritemtype} of Corollary \ref{infinitetype}, characterizes types which have this convexity property.

The structure of the paper is as follows.  In Section \ref{s-exploring}, we
state the question of the paper and reduce it from one about curves to one about types.  In Section \ref{s-preliminaries}, we obtain some basic results on
o-minimal types, following \cite{Marker86}. In Section \ref{s-scale}, we define
scale, coming from some concepts in \cite{MaSt94}, and define the notion of a
decreasing type, which makes scale more useful. Section \ref{goodbound} defines
a set of points that will be in any closed set containing a given type. Finally, in Section \ref{s-theorem} we prove Theorem \ref{maintheorem}, and get Theorem \ref{curvetheorem} as a corollary.

Throughout, we fix an o-minimal structure, $M$, expanding a real closed field, with language $L$ expanding $(<,+,\cdot,0,1)$.  All structures are  assumed to be embedded in a monster model, $\mathcal C$, in which lie all elements and sets.  ``Definable'' means ``definable with parameters in $M$.''  Tuples (of elements or functions) will be indicated by a bar above the symbol.  Subscripts, like $x_i$, indicate the $i$-th coordinate of $\bar x$, and $\bar x_{<i}$ is the tuple $\langle x_1,\ldots,x_{i-1}\rangle$. Similarly for $\bar x_{\le i}$, $\bar x_{>i}$, and $\bar x_{\ge i}$.  We will do the same for a function with image in $M^n$, writing $\gamma_i$ to mean the $i$-th component of $\bar\gamma$.  We let $\pi_{<i}$ denote projection onto the first $i-1$ coordinates. Given a function $f:\mathcal C^{n+1}\to\mathcal C^k$ and an $n$-tuple $\bar c$, if $f(\bar c,-)$ is injective, we define $f^{-1}_{\bar c}(x)$ to be the unique $y$ such that $f(\bar c,y)=x$.  For $A\subseteq\mathcal C^{m+n}$ and $\bar a\in\pi_{\le m}(A)$, let $A_{\bar a}=\{\bar y\in\mathcal C^n\scht \langle\bar a,\bar y\rangle\in A\}$.

A ``curve'' is a continuous (though not necessarily definable) map from $(0,t_0)\cap M$ to $M^n$ for some $n$ and some $t_0\in M_{>0}$.  We denote the topological closure of a set $A$ by $\cl(A)$.  If $A$ is a set, $\Pr(A)$ is the prime model of the theory of $M$ containing $A$.  If $N$ is a model and $A$ is a set, $N\langle A\rangle$ denotes $\Pr(N\cup A)$.

\section{Reducing to types}\label{s-exploring}

\begin{definition} Let $\bar\gamma$ be a (not necessarily definable) curve in $M^n$.  Say that $\bar\gamma$ is \emph{non-oscillatory} if, for each definable function $f:M^n\to M$, there exists $t_f\in M_{>0}$ such that either $f(\bar\gamma(t))=0$ for all $t\in
(0,t_f)$ or $f(\bar\gamma(t)) \ne 0$ for all $t\in (0,t_f)$.
\end{definition}

We are now equipped to ask Question \ref{mainquestion} -- given a
non-oscillatory curve and a bounded definable function, is there a definable set containing an initial segment of the curve on which the function is continuous and extends continuously to the closure of the set?

We examine the behavior of a non-oscillatory curve in $M$ more closely.

\begin{definition}
Let $\bar\gamma$ be a non-oscillatory curve in $M^n$.  Let $\tp(\bar\gamma/M)$ denote
\begin{equation*}
\{\vf(\bar x)\in L(M)\scht\text{There is $s(\vf)\in M_{>0}$ such that for all $t\in(0,s(\vf))$, }M\models\vf(\bar\gamma(t))\}.
\end{equation*}
\end{definition}

\begin{lemma}\label{lemFunctionType}
$\tp(\bar\gamma/M)$ is a complete $n$-type over $M$.
\end{lemma}

\begin{proof}
For any finite set of formulas $\vf_1,\ldots,\vf_m\in\tp(\bar\gamma/M)$, let
$s=\min_{i\le m}\{s(\vf_i)\}$.  Then for $t\in(0,s)$, we have
$M\models\vf_i(\bar\gamma(t))$ for $i=1,\ldots,m$.  This implies consistency.
It remains to show completeness.  Consider any formula, $\vf(\bar x)$.  By cell
decomposition, $\vf$ is equivalent to a disjunction of cell definitions, say
$\bigvee_{i=1}^m \bar x\in C_i$.  We may suppose by induction on $n$ that
$\exists x_n\vf(\bar x)$ is determined by $\tp(\bar\gamma/M)$.  If it is not in
$\tp(\bar\gamma/M)$, then clearly $\vf$ is not either, so we may suppose that it
is.  Since $\exists x_n\vf(\bar x)$ defines the set
$\bigvee_{i=1}^m\pi_{<n}(C_i)$, we must have that $\bar\gamma_{<n}(t)$ lies in
$\pi_{<n}(C_i)$ for some $i\le m$ and all $t\in(0,s)$, for some $s\in M_{>0}$.
Let the $n$-th coordinate cell definition of $C_i$ be given by $(f^i,g^i)$ (if the $n$-th coordinate is given by $\{f_i\}$, the argument is similar).  If
the ordering of $\gamma_n$ in the set $\{f^i(\bar\gamma_{<n}),g^i(\bar\gamma_{<n})\}$ is determined, then we are done.  But $\bar\gamma$ is non-oscillatory, which is sufficient.
\end{proof}

\begin{lemma}\label{questionequiv}
Let $\bar\gamma$ be a non-oscillatory curve in $M^n$.  The following conditions
are equivalent:
\begin{enumerate}
\item For any bounded definable function $F$ defined on an initial segment of $\bar\gamma$, there exists a definable set $C$ containing an initial segment of $\bar\gamma$ such that $F\restrict C$ is continuous and extends continuously to $\cl(C)$.
\item For any bounded definable function $F$ defined on a realization of $\tp(\bar\gamma/M)$, there exists a definable set $C$ containing $\tp(\bar\gamma/M)$ such that $F\restrict C$ is continuous and extends continuously to $\cl(C)$.
\end{enumerate}
\end{lemma}
\begin{proof}
By the definition of $\tp(\bar\gamma/M)$, for $C$ any definable set, $\tp(\bar\gamma/M)\proves\bar x\in C$ if and only if $\bar\gamma((0,s))\subseteq C$ for some $s\in M_{>0}$. Then apply this with $C$ the desired definable set containing an initial segment of $\bar\gamma$, or conversely containing $\tp(\bar\gamma/M)$.
\end{proof}

We can then reformulate Question \ref{mainquestion} for types: given a type $p$, is it true that for every bounded definable function $F$ defined on $p$, there is a definable set $C$ containing $p$ such that $F\restrict C$ continuously extends to $\cl(C)$?  But it is not hard to construct an example where no such $C$ exists.

\begin{example}\label{nearscaleexample}
Let $M=(\RR,+,\cdot,<,0,1)$, the reals as an ordered field. Let $p(x_1,x_2)$ be the type generated by the formulas $0<x_1<a$, $0<x_2<ax_1$, and $ax_1^q<x_2$, for $a\in\RR_+$, $q\in\QQ_{>1}$.  Let $F(x_1,x_2)$ be the function $\min(x_2/x_1,1)$ defined on the open first quadrant.
\end{example}

\begin{claim}
The type $p$ is consistent and complete, and if $D$ is any definable set containing the realizations of $p$, then $F\restrict D$ does not extend continuously to $\cl(D)$.
\end{claim}
\begin{proof}
We leave verification of $p$'s consistency and completeness as routine.  For the last statement, we may suppose that $D$ is a cell. Note that $D$ must be open.  Let the cell definition of $D$ be given by $(f_1,g_1)$, $(f_2,g_2)$, where $f_1$ and $g_1$ are constants.  Note that $f_1\le 0$.  By Theorem 4.6 of \cite{Miller94}, $f_2(x_1)$ and $g_2(x_1)$ asymptotically approach rational
powers of $x_1$ as $x_1$ goes to $0$.  Since $p$ requires that $x_2$ is greater than
$x_1^q$ for any rational $q>1$, the function $g_2(x_1)$ must approach a rational power of $x_1$ with exponent at most $1$.  Similarly, since $p$ requires that $x_2$ is less than $ax_1$ for any positive $a\in\RR$, the function $f_2(x_1)$ must approach a rational power of $x_1$ with exponent greater than $1$.  But then $F(x_1,f_2(x_1))$ and
$F(x_1,g_2(x_1))$ have different limits as $x_1$ goes to $0$.  Since the sets
$\{\bar x\scht x_1\in\pi_1(D)\land x_2=f_2(x_1)\}$ and $\{\bar x\scht x_1\in\pi_1(D)\land x_2=g_2(x_1)\}$ are in $\cl(D)$, it is impossible for $F\restrict D$ to extend continuously to $\bar 0$ on $\cl(D)$.
\end{proof}

\begin{corollary}\label{curvecounterex}
With $M$ and $F$ as above, if $\bar\gamma$ is the curve $\langle t,-t/\ln
t\rangle$, there is no definable set $D$ containing an initial segment of $\bar\gamma$ on which $F\restrict D$ is continuous and extends continuously to $\cl(D)$.
\end{corollary}
\begin{proof}
$\tp(\bar\gamma/M)$ is $p$ in Example \ref{nearscaleexample}.
\end{proof}

We may ask, then, for necessary and sufficient conditions on $p$, an $n$-type in an o-minimal field, so that, for any $F$, a bounded definable function
 on $p$, there is a definable set $C$ containing $p$ such that $F$ is continuous on $C$ and $F\restrict C$ extends continuously to $\cl(C)$.  In order to characterize such types, we will need to extend a classification of o-minimal types developed by \cite{MaSt94}.

\section{O-minimal background}\label{s-preliminaries}

Before we begin to present any new machinery, we will need to state some basic
results that follow from \cite{Marker86} and \cite{vdD98book}.  We use here the results of \cite{Marker86} but follow some of the terminology of \cite{Tressl05a}: the definable non-algebraic $1$-types are called ``principal.''  To each principal type over a set $A$ is associated a unique element $a\in\dcl(A)\cup\{\pm\infty\}$ to which it is ``closest.''  We say that a principal type is ``principal above/below/near $a$.''  The results of
\cite{Marker86} and \cite{vdD98book} will be used freely -- the reader is referred there for background.

\begin{lemma}\label{boundabove}
Let $c_1,c_2$ be principal over $A$, near $\beta_1,\beta_2\in
\dcl(A)\cup\{\pm\infty\}$ respectively.  If $c_1$ is non-principal over $c_2A$, then there is some $A$-definable function $f(x)$ such that
$\lim_{x\to\beta_1}f(x)=\beta_2$ and $c_2$ lies between $f(c_1)$ and
$\beta_2$.
\end{lemma}
\begin{proof}
We suppose that $c_1,c_2$ are above finite $\beta_1,\beta_2$, respectively -- the proof is similar for the other possibilities.  Since $c_1$ is non-principal over
$c_2A$, there is some $A$-definable $g$ such that $\beta_1<g(c_2)<c_1$.  Let $f(x)=g^{-1}(x)$.  We show that $g$ is increasing on an interval above $\beta_2$.  If it were constant, this would imply that $g(c_2)$ is $A$-definable, which would contradict $c_1$ being principal over $A$.  If it were decreasing, we could restrict $g$ to an interval above $\beta_2$ on which it was continuous and decreasing, let $\delta$ be the right endpoint of the image of $g$, and then consider $f((\beta_1+\delta)/2)$, which would lie between $\beta_2$ and $c_2$, contradicting $c_2$ being principal over $A$. Then $f(c_1)>c_2$. Similarly, $\lim_{x\to\beta_2^+}g(x)=\beta_1$, or else either this limit or $\lim_{x\to\beta_1^+}f(x)$ would contradict either $c_1$ or $c_2$ being principal over $A$, respectively.  Thus $\lim_{x\to\beta_1^+}f(x)=\beta_2$.
\end{proof}

\begin{lemma}\label{fiberclosure}
Let $S$ be a definable set in $\mathcal C^{m+n}$.  Let $S'=\{\bar x\in\mathcal C^{m+n}\scht\exists\bar a\in\pi_{\le m}(S)(\bar x\in\cl(\{\bar a\}\times S_{\bar a}))\}$.  Then there is a partition of $\mathcal C^m$ into definable subsets $A_1,\ldots,A_k$ such that $S'\cap(A_i\times \mathcal C^n)=\cl(S)\cap (A_i\times\mathcal C^n)$, for $i=1,\ldots,k$.  In other words, the closure of a fiber is the fiber of the closure.
\end{lemma}
\begin{proof}
  $S'$ and $\cl(S)$ satisfy the conditions of Corollary 2.3, Chapter
  6, of \cite{vdD98book}, with $A=\mathcal C^m$, so we can find
  $A_1,\ldots,A_k$ such that $S'\cap(A_i\times\mathcal C^n)$ is closed
  in $\cl(S)\cap(A_i\times\mathcal C^n)$, which implies that the two
  sets are equal, for each $i=1,\ldots,k$
\end{proof}

\section{Scale and decreasing types}\label{s-scale}

The notion of a ``region'' in the discussion after Question \ref{mainquestion} is closely related to a concept that was first defined in \cite{MaSt94}, although not formally named.

\begin{definition}
Let $p=\tp(a/B)$ be non-principal, with $A\subset B$. Let $p$ be \emph{out of scale on $A$} if for every unary $B$-definable function $f$, the set $f(\Pr(A))$ is neither cofinal nor coinitial at $a$ in $\Pr(B)$.
\end{definition}

Marker and Steinhorn \cite{MaSt94} obtained the following theorem.

\begin{theorem}\label{ms9421}\textup{(\cite{MaSt94}, Theorem 2.1)}
Let $p\in S_n(M)$.  Then $p$ is definable if and only if for any $\bar c$ realizing $p$, $M\langle\bar c\rangle$ realizes only principal types over $M$.
\end{theorem}

\begin{lemma}\label{outofscaledfbl}
Let $A$ be a set and $p\in S_n(A)$ an $n$-type, with $\bar c\models p$.  If $\tp(c_i/A\bar c_{<i})$ is principal, algebraic, or out of scale on $A$ for $i=1,\ldots,n$, then $p$ is definable.
\end{lemma}
\begin{proof}
Suppose that $p$ is an $n$-type and not definable.  Let $\bar c\models p$ and let $M=\Pr(A)$.  Let $i$ be the first coordinate such that $\tp(\bar c_{\le i}/M)$ is not definable.  Then $\tp(c_i/M\bar c_{<i})$ is not principal or algebraic by Lemma 2.5 of \cite{MaSt94}. By Lemma 2.7 of \cite{MaSt94}, there is an $M\bar c_{<i}$-definable function $f$ such that $\tp(f(c_i)/M)$ is non-principal.  Since $\tp(\bar c_{<i}/M)$ is definable by choice of $i$, Theorem \ref{ms9421} implies that $M\langle\bar c_{<i}\rangle$ realizes no elements in $\tp(f(c_i)/M)$.  Thus $f^{-1}(M)$ is cofinal and coinitial at $c_i$ in $M\langle\bar c_{<i}\rangle$, and so $\tp(c_i/A\bar c_{<i})$ is not out of scale on $A$.
\end{proof}

\subsection*{Decreasing types}
Given an $n$-type, the ordering of the variables can affect the type of each variable over the preceding ones.  Consider the type of $\langle\epsilon,\epsilon'\rangle$ over $M=(\RR,+,\cdot,<)$, where $1\gg\epsilon\gg\epsilon'>0$.  We have that $\tp(\epsilon/M)$ and $\tp(\epsilon'/M\epsilon)$ are principal.  However, if we
consider the elements in reverse order, $\tp(\epsilon'/M)$ is still principal,
but now $\tp(\epsilon/M\epsilon')$ is non-principal.  We wish to fix a class of
orderings of $p$'s coordinates that will provide some predictability.

We begin by defining a useful partial ordering.

\begin{definition}
Let $A$ be a set.  Define $a\precsim_A b$ if $\dcl(aA)$ is coinitial in $\dcl(bA)$ above $0$.
\end{definition}

Note that $\precsim_A$ defines a partial ordering, since ``coinitiality'' is transitive.

\begin{definition}
Given a base set, $A$, and a tuple, $\bar c=\langle c_1,\ldots,c_n\rangle$, define $c_j\precsim_i c_k$, for $i\le j,k\le n$, if $c_j\precsim_{A\bar c_{<i}}c_k$.  Given an $n$-type, $p$, define $x_j\precsim_i x_k$ if, for some (equivalently, every) realization $\bar c$ of $p$, we have $c_j\precsim_i c_k$.
\end{definition}

\begin{lemma}\label{decreasing}
Let $p$ be an $n$-type over a set $A$.  Then there exists a reordering of the variables of $p$ such that, in the new ordering, $x_i\succsim_i x_j$, for all $i<j\le n$.
\end{lemma}
\begin{proof}
We reorder $p$ in stages.  At stage $i$, having determined $\bar x_{<i}$, there is at
least one maximal element in the partial order $\precsim_i$ among the remaining $x_j$.  Set
any such maximal element to be $x_i$.
\end{proof}

\begin{definition}\label{defni}
If the variables of $p$ satisfy the conclusion of Lemma \ref{decreasing}, we say that $p$ is \emph{decreasing}.  For $i$ an index in the variables of $p$, let $Q(i)$ denote the greatest index at most $i$ such that $\tp(c_{Q(i)}/\bar c_{<Q(i)}A)$ is principal, and $0$ if such index does not exist.
\end{definition}

There is a connection between decreasing sequences and the $T$-convex subrings of \cite{vdDLew95}.  If $\bar c$ is a decreasing sequence over $A$, then for $1\le j<k$, the convex hull of $\Pr(A\bar c_{<j})$ is a $T$-convex subring contained in the
convex hull of $\Pr(A\bar c_{<k})$, with equality if and only if
$Q(k)\le j$.  The connections between decreasing sequences and
$T$-convex subrings and valuations will be presented in a future paper.

\begin{lemma}\label{noncutsafterni}
Let $p$ be a decreasing $n$-type over a set $A$, let $\bar c\models p$, and let $k$ be an index such that $\tp(c_k/A\bar c_{<k})$ is principal.  Then for $i\ge k$, $\tp(c_i/\bar c_{<k}A)$ is principal.
\end{lemma}
\begin{proof}
Since $c_k\succsim_k c_i$ (by definition of ``decreasing''), we know
that $\dcl(c_iA\bar c_{<k})$ is coinitial above $0$ in $\dcl(A\bar c_{\le k})$.
Since $c_k$ is principal over $A\bar c_{<k}$, there is some $d\in\dcl(A\bar
c_{\le k})$, principal above $0$ over $A\bar c_{<k}$.  By coiniality, there is
some $d'\in\dcl(c_iA\bar c_{<k})$, with $0<d'<d$, but then $d'$ witnesses that
$c_i$ is principal over $A\bar c_{<k}$.
\end{proof}

Note that then $\tp(c_i/A\bar c_{<Q(i)})$ is principal.

\begin{lemma}\label{cknotprincipal}
For $i\le n$ and $k=Q(i)>0$, the type $\tp(c_k/A\bar c_{<k}c_i)$ is not principal.
\end{lemma}
\begin{proof}
We first observe that $\tp(c_i/A\bar c_{\le k})$ is non-principal. Otherwise, by results of \cite{Marker86}, for some $j\in(k,i]$, we would have $\tp(c_j/A\bar c_{<j})$ principal, contradicting the definition of $k=Q(i)$. It follows that $\tp(c_k/A\bar c_{<k}c_i)$ is non-principal as well.
\end{proof}

\section{Good bounds and $i$-closures}\label{goodbound}

Given an $n$-tuple, $\bar c$, and set $A$, there are certain points that must be in any closed $A$-definable set containing $\bar c$, namely the $i$-closures defined in this section. The $i$-closure of $\bar c$ for each $i\le n$ is the limit (in the sense of \cite{HrLo11}) of $\tp(\bar c_{\ge Q(i)}/Ac_{<Q(i)})$. A principle of our proof of Theorem A will be that if a function can be continuously extended to the $i$-closure points for all $i$, then it can be continuously extended to an $A$-definable closed set containing $\bar c$. We can assure continuity on $i$-closures by bounding the various values a function takes by another function that goes to $0$ as it approaches an $i$-closure point. These functions are the ``good bounds.''

When we prove Theorem \ref{maintheorem}, we will prove it just for types of a certain form.  We will then show that all other kinds of types can be transformed into this form.  In this section, therefore, we restrict to considering only this kind of type.

\begin{condition}\label{uniquecond}
The type $p$ is a decreasing $n$-type over a set $A$, the tuple $\bar c$ is a realization of $p$.  We have $i\le n$ and $k=Q(i)>0$.  For $j=k,\ldots,n$, the type $\tp(c_j/\bar c_{<k}A)$ is principal above finite $\beta_j(\bar c_{<k})\in A\langle\bar c_{<k}\rangle$.  Let
$\bar\beta=\langle\beta_k,\ldots,\beta_n\rangle$.
\end{condition}

Note that the $A$-definable functions $\beta_j$ depend on the value of $i$.  If Condition \ref{uniquecond} is true for some $\bar c$, it is true for any $\bar c'\models p$, and thus is a condition just on $p$, $A$, and $i$.  Note that, for any $p$ a decreasing type over $A$ and $\bar c\models p$ with $k=Q(i)$ for some coordinate $i$ and $j\ge k$, we know $\tp(c_j/\bar c_{<k}A)$ is principal by Lemma \ref{noncutsafterni}.

\begin{definition}
Suppose Condition \ref{uniquecond} holds.  Then, for any tuple $\bar a$ with
length at least $k-1$ such that $\bar\beta$ is defined on $\bar a_{<k}$, let
\begin{equation*}
\icl(i,\bar a)=\langle \bar a_{<k},\bar\beta(\bar a_{<k})\rangle.
\end{equation*}
We also call $\icl(i,\bar a)$ the \emph{$i$-closure} of $\bar a$. Note that $\icl(i,\bar a)\in\dcl(A\bar a_{<k})$.
\end{definition}

If $C$ is a definable set, then $\icl(i,C)=\{\icl(i,\bar x)\scht\bar x\in C\}$. If $p$ is a decreasing type, but $i$ is such that Condition \ref{uniquecond} does not hold, then we set $\icl(i,C)=\emptyset$.

\begin{lemma}
Suppose Condition \ref{uniquecond} holds.  If $\tp(c_i/A\bar c_{<i})$ is non-principal, then $\icl(i,\bar x)=\icl(i-1,\bar x)$.
\end{lemma}
\begin{proof}
Since $\tp(c_i/A\bar c_{<i})$ is non-principal, $Q(i)<i$. Now the conditions on $Q(i)$ and $Q(i-1)$ are the same, so $Q(i)=Q(i-1)$, and thus
\begin{equation*}
\icl(i,\bar x)=\langle \bar x_{<Q(i)},\beta_{Q(i)}(\bar
x_{<Q(i)}),\ldots,\beta_n(\bar x_{<Q(i)})\rangle =\icl(i-1,\bar x).
\end{equation*}
\end{proof}

There is a $A$-definable set on which the $i$-closures are distinguished.

\begin{lemma}\label{uniqueclosure}
If Condition \ref{uniquecond} holds, then there is an $A$-definable set $C^0$ containing $\bar c$ such that, for every $\bar a\in\pi_{<k}(C^0)$, the set $\cl(C^0)$ contains a unique point, $\bar
d$, with $\bar d_{\le k}=\langle \bar a,\beta_k(\bar a)\rangle$. Moreover, for
each $\bar a$ (and in particular for $\bar c_{<k}$), this point is independent
of choice of $C^0$ -- in fact, it is $\icl(i,\bar a)$.
\end{lemma}
\begin{proof}
By Lemma \ref{cknotprincipal} and Lemma \ref{boundabove}, for each $j>k$ there is some $A$-definable
$k$-ary function, $h_j$, such that
\begin{equation}\label{eqhprops}
\begin{split}
c_j<h_j(\bar c_{\le k}),\text{ and}\\
\lim_{y\to\beta_k(\bar c_{<k})}h_j(\bar c_{<k},y)=\beta_j(\bar c_{<k}).
\end{split}
\end{equation}

Let $C$ be a $A$-definable set containing $\bar c$ such that: $\bar\beta$ is continuous on $C$; $h_j>\beta_j$ for $j>k$ (possible since
$h_j(\bar c_{\le k})>\beta_j(\bar c_{<k})$); and \eqref{eqhprops} holds on all
of $C$ with $\bar c$ replaced by $\bar x$ (possible since it holds for $\bar c$
-- note that the limit statement is first-order).  Let
\begin{equation*}
B=\{\bar x\in C\scht x_j\in(\beta_j(\bar x_{<k}),h_j(\bar x_{\le k})),\text{ for
}j>k\}.
\end{equation*}
Note that $\bar c\in B$.  By Lemma \ref{fiberclosure}, we can decompose $B$ into definable
sets, $C^0,\ldots,C^r$, on each of which, for any $\bar a\in\pi_{<k}(C^s)$,
we have $\cl(C^s_{\bar a})={\cl(C^s)}_{\bar a}$ -- the closure of a fiber is the
fiber of the closure.  Without loss of generality, let $C^0$ be the cell
containing $\bar c$.

Let $\bar a\in\pi_{<k}(C^0)$.  Let $D=\{\bar a\}\times C^0_{\bar a}$.  Let
$\bar d\in\cl(C^0)$, with $\bar d_{\le k}=\langle \bar a,\beta_k(\bar
a)\rangle$.  Note that this implies $\bar d\in\cl(D)$.  We want to show that
$\bar d=\icl(i,\bar a)$.  For $j>k$, we have $d_j\ge\beta_j(\bar a)$.  Let $\bar\gamma(t)$ be an $A\bar a$-definable curve in
$D$, with $\lim_{t\to 0}\bar\gamma(t)=\bar d$, whose existence is guaranteed by
Corollary 1.5 of Chapter 6 in \cite{vdD98book}.  For $j>k$,
\begin{equation*}
d_j\le\lim_{t\to 0^+}h_j({\bar\gamma(t)}_{\le k})=\lim_{y\to\beta_j(a)^+}h_j(\bar
a,y)=\beta_j(\bar a).\\
\end{equation*}
Thus, $\bar d=\langle \bar a,\bar\beta(\bar a)\rangle=\icl(i,\bar a)$.
\end{proof}

When Condition \ref{uniquecond} holds, $C^0$ always denotes the set coming
from Lemma \ref{uniqueclosure}.

\begin{lemma}\label{onlyicl}
Let $p$ be a decreasing $n$-type over $\emptyset$, and $C$ any $\emptyset$-definable set containing $p$. There is a $\emptyset$-definable set $D$ containing $p$ such that $\cl(D)\setminus C\subseteq\bigcup_{i\le n}\icl(i,D)$.
\end{lemma}
\begin{proof}
Let $C^{0,i}$ be the set $C^0$ for $i$ if Condition \ref{uniquecond} holds for $p$ and $i$, and $C$ otherwise. Let $D=\bigcap_{i\le n} C^{0,i}$. The set $D$ is non-empty, since it contains all realizations of $p$. We may restrict $D$ and suppose that it is a cell. Let $\bar c\models p$. For $i=1,\ldots,n$, if $\tp(c_i/\bar c_{<i})$ is algebraic, we may suppose that the cell definition of $D$ at $i$ is just the function $f_i$ defining $c_i$ from $\bar c_{<i}$. We now refine the definition of $D$, coordinate by coordinate. Let $(f_i,g_i)$ be the $i$-th coordinate cell definition of $D$. If $\tp(c_i/\bar c_{<i})$ is principal, say above $\alpha(\bar c_{<i})$, then we may replace $f_i$ by $\alpha$ and $g_i$ by $(\alpha+g_i)/2$. If $\tp(c_i/\bar c_{<i})$ is non-algebraic and non-principal, then we can find $\emptyset$-definable functions $f'_i<g'_i$ lying in $(f_i,g_i)$, and replace $f_i$ and $g_i$ with those.

Now if $\bar a\in\cl(D)\setminus C$, let $i$ be the least coordinate such that $\bar a_{\le i}\notin\pi_{\le i}(C)$. Then $\tp(c_i/\bar c_{<i})$ is clearly principal, and $a_i=f_i(\bar a_{<i})$, supposing without loss of generality that $c_i$ is principal above $f_i(\bar c_{<i})$. By Lemma \ref{uniqueclosure}, since $Q(i)=i$, there is exactly one point in $\cl(D)$ with first $i$ coordinates $\langle \bar a_{<i},f_i(\bar a_{<i})\rangle$, namely $\icl(i,\bar a_{<i})$. Taking any $\bar a'\in D$ with $\bar a'_{<i}=\bar a_{<i}$, we see that $\bar a=\icl(i,\bar a')\in\icl(i,\bar D)$.
\end{proof}

\begin{convention}
For the rest of this paper, if $f:\mathcal C^i\to\mathcal C$ is a function, and $n>i$, we abuse
notation and write $f(\bar x)$ for $\bar x\in\mathcal C^n$ to mean $f(\pi_{\le i}(\bar x))$, and consider $f$ a function on $\mathcal C^n$ when convenient.
\end{convention}

\begin{definition}
Suppose Condition \ref{uniquecond} holds.  Let $f$ be an $i$-ary $A$-definable bounded function such that, for some $A$-definable $C\subseteq C^0$ with $\bar c\in C$, the function $f$ is continuous and non-negative on $C\cup\icl(i,C)$, and moreover $f(\icl(i,C))=0$.  Then we call $f$ a \emph{good bound at $i$}.  \end{definition}

\begin{definition}
Suppose Condition \ref{uniquecond} holds. Let
\begin{equation*}
m_i(\bar x_{\le i})=\min(|x_{Q(i)}-\beta_{Q(i)}(\bar x_{<Q(i)})|,1).
\end{equation*}
\end{definition}

\begin{lemma}
If Condition \ref{uniquecond} holds, then $m_i$ is a good bound at $i$, with domain $C^0$.
\end{lemma}
\begin{proof}
The function $\beta_{Q(i)}$ is continuous on $\pi_{<Q(i)}(C^0)$ by definition of
$C^0$, hence $m_i$ is continuous on $C_0\cup\icl(i,C^0)$.  If $\bar x\in C^0$,
then $\icl(i,\bar x)_{Q(i)}=\beta_{Q(i)}(\bar x_{<Q(i)})$, so $m_i$ is $0$ on $\icl(i,C^0)$.
\end{proof}

\begin{lemma}\label{goodbounddown}
Let Condition \ref{uniquecond} hold. Suppose that Condition \ref{uniquecond} also holds for $i-1$. If $f$ is a good bound at $i$ then there
exists $f'$ with $f'\ge f$ on some definable set containing $\bar c$, and
$f'$ a good bound at $i-1$.
\end{lemma}
\begin{proof}
By the definition of a good bound, there is some $A$-definable $C\subseteq C^0$ such that $f$ is continuous and non-negative on $C\cup\icl(i,C)$ with $f(\icl(i,C))=0$. If $\tp(c_i/\bar c_{<i}A)$ is algebraic, then we can define $f'$ to equal $f$ on a definable set containing $\bar c$, so we may suppose not.
 
\subsubsection*{Case 1: $\tp(c_i/\bar c_{<i}A)$ is principal}

The element $c_{i-1}$ is principal over $M\langle \bar c_{<Q(i-1)}\rangle$ near
$\beta_{i-1}(\bar c_{<Q(i-1)})$, where $\beta_{i-1}$ is a finite $A$-definable function (note that $\beta_{i-1}$ is not part of the original sequence of functions, $\bar\beta$).  Without loss of generality $\tp(c_{i-1}/A\bar c_{<Q(i-1)})$ is principal above $\beta_{i-1}(\bar c_{<Q(i-1)})$.  We may restrict $C$ so that $\icl(i-1,\bar x)\notin C$ for $\bar x\in C$, since we can take $C$ to have lower boundary at least $\beta_{i-1}(\bar x_{<Q(i-1)})$ at the $(i-1)$-st coordinate.  For $\bar x\in C$, we have $\icl(i-1,\bar x)\ne\icl(i,\bar x)$, since $\icl(i-1,\bar x)_{i-1}=\beta_{i-1}(\bar x_{<Q(i-1)})<x_{i-1}=\icl(i,\bar x)_{i-1}$.  Since $f$ is a good bound at $i$, we know that $f(\icl(i,\bar x))=0$
for $\bar x\in C$, and therefore $f(\icl(i,\bar x))<m_{i-1}(\icl(i,\bar x))$.
Note that $Q(i)=i$, since $\tp(c_i/\bar c_{<i}A)$ is principal.  Since $f$ and $m_{i-1}$ are continuous, there is some $A$-definable function $h(\bar x_{<i})$ such that, if $x_i\in(\beta_i(\bar
x_{<i}),h(\bar x_{<i}))$, then $f(\bar x)<m_{i-1}(\bar x)$.  Restrict $C$ to have upper boundary at most $h$ on the $i$-th coordinate.  Then, on our new $C$, we have $m_{i-1}>f$, and $m_{i-1}$ is a good bound at $i-1$.

\subsubsection*{Case 2: $\tp(c_i/\bar c_{<i}A)$ is non-principal}

Since $f(\icl(i,\bar c))=0$, there is an $A\bar c_{<Q(i)}$-definable continuous increasing function, $\delta(t)$, such that given any sufficiently small $\epsilon>0$, we have $f(\bar x)<\epsilon$ for $x\in C(\epsilon)$, with
\begin{equation*}
C(\epsilon)=\{\bar x\in C\scht \bar x_{<Q(i)}=\bar c_{<Q(i)}\land|\bar x-\icl(i,\bar c)|<\delta(\epsilon)\}.
\end{equation*}

(Here and going forward, $|\cdot|$ is the $\sup$ norm.) By the proof of Lemma \ref{cknotprincipal}, $\tp(c_j/A\bar c_{\le Q(i)})$ is non-principal for $j\in[Q(i),i]$. Thus Lemma \ref{boundabove} implies that each $|c_j-\beta_j(\bar c_{<Q(i)})|$ is bounded by some $A\bar c_{<Q(i)}$-definable function of $c_{Q(i)}$, say $h_j(c_{Q(i)})$, with $\lim_{t\to\beta_{Q(i)}}h_j(t)=0$. Let $h(t)=\max_{j\in[i,Q(i)]}h_j(t)$.

Define $g(\bar x)$ to be $\sup\{f(\bar x_{<i},t)\scht|\langle\bar x_{<i},t\rangle-\icl(i,\bar x)_{\le i}|<h(x_{Q(i)})\}$, with domain $\{\bar x\in C\scht \bar x_{<Q(i)}=\bar c_{<Q(i)},|\bar x_{\le i}-\icl(i,\bar x)_{\le i}|<h(x_{Q(i)})\}$. If $\bar x\in\dom(g)$, then $\bar x_{\le i}\in \pi_{\le i}(C(\delta^{-1}(h(x_{Q(i)}))))$. We restrict $C$ to a set containing $\bar c$ such that $g$ is continuous on its domain. Since the value of $f$ on $C(\epsilon)$ goes to $0$ as $\epsilon$ goes to $0$, the value of $g(\bar x)$ as $\bar x$ approaches $\icl(i,\bar c)$ must likewise go to $0$, since for $\bar x$ with $\delta^{-1}\circ h$ defined on $x_{Q(i)}$, $g$ is bounded above by $\delta^{-1}(h(x_{Q(i)}))$, and this value goes to $0$ as $x_{Q(i)}$ goes to $\beta_{Q(i)}(\bar x_{<Q(i)})$. As well, $g\ge f$ on $\dom(g)$.

The above argument also holds with the parameters $\bar c_{<Q(i)}$ replaced by elements in an $A$-definable neighborhood of $\bar c_{<Q(i)}$, say $D$, with $h$ now a $Q(i)$-ary function. This shows that $g$ is a good bound at $i-1$ on the set
\begin{equation*}
\{\bar x\in C\scht \bar x_{<Q(i)}\in D\land|\bar x-\icl(i,\bar x)|<h(\bar x_{\le Q(i)})\}.
\end{equation*}
Redefining $C$ to have $i$-th coordinate upper boundary function $h$, we see that $g\ge f$ on $C$.
\end{proof}

\section{Main result}\label{s-theorem}

We are now ready to prove our main theorem.  We restate it, since all terms have finally been defined.

\begin{theorem}\label{maintheoremrestate}
Let $A\subseteq M$ be a set.  Let $p$ be a finite decreasing $n$-type over $A$.  Let $\bar c=\langle c_1,\ldots,c_n\rangle\models p$.  The following statements are equivalent:
\renewcommand{\theenumi}{S\arabic{enumi}}
\begin{enumerate}
\item\label{theoremitemfunction} For every $A$-definable bounded $n$-ary function, $F$, defined on $\bar c$, there is an $A$-definable set $C$ with $\bar c\in C$, such that $F\restrict C$ is continuous and extends continuously to $\cl(C)$.
\item\label{theoremitemtype} There is $i_0\le n$ such that $\tp(c_i/A\bar c_{<i})$ is algebraic, principal, or out of scale on $A\bar c_{<Q(i)}$ for $i=i_0,\ldots,n$, and for $i<i_0$, $\tp(c_i/A)$ is non-principal.
\end{enumerate}
\renewcommand{\theenumi}{\arabic{enumi}}
\end{theorem}
\begin{proof}
We suppose in the proof that $p$ is independent, and satisfies Condition \ref{uniquecond} for all $i\le n$ such that $Q(i)>0$. Afterward we will show how to reduce other cases to this one.

We first prove that \eqref{theoremitemtype} implies \eqref{theoremitemfunction}.\footnote{Thanks to P. Speissegger for the use in this proof of van den Dries' result on fiberwise-continuous functions.} We will go by induction on $n$, although we will also have an additional ``inner'' induction.  By adding constants for the elements of $A$ to the language $L$, we may assume that $A=\emptyset$.  Let $P=\Pr(\emptyset)$.

Let $F$ be continuous on $D$, an open $\emptyset$-definable cell containing $p$. We suppose that $D$ satisfies the conclusion of Lemma \ref{onlyicl} for some definable set $D'\supseteq D$, so $F$ is already continuous on $\cl(D)\setminus\bigcup_{i\le n}\icl(i,D)$. Let $f_i$, $g_i$ be the $\emptyset$-definable lower and upper bounding functions in the definition of $D$ as a cell. We now construct new $\emptyset$-definable bounding functions to replace these, starting at $i=n$ and going down to $i=1$, using induction hypotheses on the boundary functions already constructed, as well as our global induction hypothesis on $n$. For any $\bar x$ with $\bar x_{\le
i}\in\pi_{\le i}(D)$, let $E^i_{\bar x}=\{\bar x_{\le i}\}\times D_{\bar x_{\le
i}}$.  

We have two induction statements at stage $i$:

\renewcommand{\theenumi}{I\arabic{enumi}}
\renewcommand{\labelenumi}{(\theenumi)}
\begin{enumerate}

\item\label{i1} For all $\bar x\in\pi_{\le i}(D)$, $F\restrict E^i_{\bar x}$ is continuous and extends continuously to $\cl(E_{\bar x}^i)$.

\item\label{i2} If $Q(i)>0$, then there is a $\emptyset$-definable $i$-ary function $G$, a good bound at $i$, such that for any $\bar a,\bar a'\in D$ with $\bar a_{\le i}=\bar a'_{\le i}$, we have $|F(\bar a)-F(\bar a')|\le G(\bar a_{\le i})$.

\end{enumerate}
\renewcommand{\theenumi}{\arabic{enumi}} \renewcommand{\labelenumi}{(\theenumi)}

Both \eqref{i1} and \eqref{i2} are trivially true when $i=n$, and \eqref{i1} for $i=0$ gives \eqref{theoremitemfunction}, our desired result. We prove \eqref{i1} and \eqref{i2} for $i-1$, given them for $i$.

\begin{claim}\label{reducetofg}
We may shrink $D$ so that, for $\bar x\in\pi_{<i}(D)$, the function $F\restrict E^{i-1}_{\bar x}$ is continuous and extends continuously to $\cl(E^{i-1}_{\bar
x})\cap\{\bar y\scht\bar y_{\le i}\in\pi_{\le i}(D)\}$.
\end{claim}
\begin{proof}
Let $\oF$ be $F$ with its domain extended onto $\cl(E^i_{\bar x})$ for each
$\bar x\in\pi_{\le i}(D)$.  By \eqref{i1} and Corollary 2.4 of Chapter 6 of
\cite{vdD98book}, for any $\bar x\in\pi_{\le
i-1}(D)$, we can partition $(f_i(\bar x),g_i(\bar x))$ into intervals $I_1(\bar x),\ldots,I_{r(\bar x)}(\bar x)$ (and their endpoints) so that $\oF$ is continuous on
\begin{equation*}
\{\bar y\in \cl(D)\scht \bar y_{<i}=\bar x\land y_i\in I_j(\bar x)\},
\end{equation*}
for $1\le j\le r(\bar x)$.  Let $r=r(\bar c)$. Let $I_j(\bar c_{<i})$ be given by $(h_j(\bar c_{<i}),h_{j+1}(\bar c_{<i}))$, for some $\emptyset$-definable functions $h_j$, $j=1,\ldots,r+1$, with $h_1=f_i$ and $h_{r+1}=g_i$. Let $U\subseteq \mathcal C^{i-1}$ be a $\emptyset$-definable open set containing $\bar c_{<i}$ such that $r(\bar x)$ is constant on $U$, the functions $h_1,\ldots,h_{r+1}$ are continuous on $U$, and $I_j(\bar x)=(h_j(\bar x),h_{j+1}(\bar x))$ for all $\bar x\in U$ and $j=1,\ldots,r$.  Let $k$ be such that $c_i\in I_k(\bar c_{<i})$.  Replace $D$ by $D\cap \{\bar x\scht \bar x_{<i}\in U, x_i\in(h_k(\bar x_{<i}),h_{k+1}(\bar x_{<i}))\}$.  Then for each $\bar x\in\pi_{<i}(D)$, the function $\oF$ is continuous on
\begin{equation*}
\{\bar y\in\cl(D)\scht\bar y_{<i}=\bar x\land\bar y_{\le i}\in\pi_{\le i}(D)\}\supseteq \cl(E^{i-1}_{\bar x})\cap\{\bar y\scht
\bar y_{\le i}\in\pi_{\le i}(D)\},
\end{equation*}
as desired.
\end{proof}

Now all that remains to show \eqref{i1} for $i-1$ is to consider points in $\cl(E^{i-1}_{\bar x})\setminus\{\bar y\scht \bar y_{\le i}\in\pi_{\le
  i}(D)\}$ -- points with $i$-th coordinate equal to $f_i(\bar x_{<i})$ or
$g_i(\bar x_{<i})$.

Define $\mu:\pi_{\le i}(D)\to \mathcal C$ by $\mu(\bar x)=\sup\{F(\bar y)\scht \bar y_{\le i}=\bar x\land\bar y\in D\}$. We will use $\mu$ in applying the triangle inequality to bound differences in values of $F$. Shrinking $D$, we may suppose that $\mu$ is continuous on $D$. By \eqref{i2} for $i$, there is $G$, a good bound at $i$, such that for $\bar x\in D$, we have $|\mu(\bar x_{\le i})-F(\bar x)|\le G(\bar x_{\le i})$. We must now consider two cases. In each case, we will prove both \eqref{i1} and \eqref{i2} for $i-1$.

\subsubsection*{Case 1: $\tp(c_i/\bar c_{<i})$ is principal}

We have $\tp(c_i/\bar c_{<i})$ principal above $\beta_i(\bar c_{<i})$ over $\bar c_{<i}$ for $\beta_i$ some $\emptyset$-definable function. We may suppose that $f_i\le\beta_i$ on $\pi_{<i}(D)$, since this is true at $\bar c_{<i}$, and so we may actually suppose that $f_i=\beta_i$.  If we replace $g_i$ by $(g_i+f_i)/2$, we guarantee that, for $\bar x\in\pi_{<i}(D)$, $F$ is continuous on the set
\begin{equation*}
\{\bar y\in \cl(E^i_{\bar x})\scht f_i(\bar y_{<i})<y_i\le g_i(\bar y_{<i})\}.
\end{equation*}

Thus, to prove \eqref{i1} for $i-1$ in this case it only remains to show that $F$ extends continuously onto the points where $y_i=f_i(\bar y_{<i})$.  By Lemma \ref{uniqueclosure}, we can restrict $D$ further so that for each $\bar x\in \pi_{<i}(D)$, the point $\icl(i,\bar x)$ is the unique point in $\cl(D)$ with first $i$ coordinates $\langle \bar x,f_i(\bar x)\rangle$ (note that $Q(i)=i$). For $\bar x\in D$, define $F(\icl(i,\bar x))=\lim_{y\to f_i(\bar x_{<i})^+}\mu(\bar x_{<i},y)$. We show that this is a continuous extension of $F\restrict E^{i-1}_{\bar x}$ for each $\bar x\in\pi_{<i}(D)$.  Fix $\bar a\in\pi_{<i}(D)$.  Let $\epsilon\in M$ be any positive element. Fix $\delta>f_i(\bar a)$ such that for $y\in(f_i(\bar a),\delta)$, we have $G(\bar a,y)<\epsilon/2$ and $|\mu(\bar a,y)-F(\icl(i,\bar a))|<\epsilon/2$.  Then for any $\bar b$ in the set $U=\{\bar x\in \cl(E^{i-1}_{\bar a})\scht x_i<\delta\}$, we have $|F(\bar b)-F(\icl(i,\bar a))|\le\epsilon/2+|\mu(\bar b)-F(\icl(i,\bar a)|<\epsilon$.  The set $U$ is open in $\cl(E^{i-1}_{\bar a})$ and contains $\icl(i,\bar a)$, so $F\restrict E^{i-1}_{\bar a}$ is continuous at $\icl(i,\bar a)$.  Thus, we have satisfied \eqref{i1} for $i-1$.

We must also satisfy condition \eqref{i2} for $i-1$, supposing $Q(i-1)>0$.  Let $G'$ be a good bound at $i-1$ with $G'\ge G$ guaranteed by Lemma \ref{goodbounddown} (we may shrink $D$ so that $D$ is a valid domain for $G'$).  Define
\begin{equation*}
S(\bar x,z)=\sup\left\{y : |\mu(\bar x,y)-F(\icl(i,\bar x))|<z\right\}.
\end{equation*}
Now replace our $i$-th coordinate boundary function, $g_i$, with $\min(g_i(\bar x),S(\bar x,G'(\bar x)))$.  We have then guaranteed that applying $F$ to any point will yield a value differing by less than $G'(\bar x)+G(\bar x)$ from $F$ applied to its $i$-closure point.

Given $\bar a,\bar a'\in D$ with $\bar a_{<i}=\bar a'_{<i}$, we have
\begin{multline*}
4G'(\bar a_{<i}) \ge
G(\bar a_{\le i})+G(\bar a'_{\le i})+|\mu(\bar a_{\le i})-F(\icl(i,\bar a_{<i}))|+|\mu(\bar a'_{\le i}))-F(\icl(i,\bar
a_{<i}))| \ge\\
|F(\bar a)-\mu(\bar a_{\le i})|+|F(\bar a')-\mu(\bar a'_{\le i})|+|\mu(\bar
a_{\le i})-\mu(\bar a'_{\le i})| \ge
|F(\bar a)-F(\bar a')|.
\end{multline*}
Thus, since $4G'$ is a good bound at $i-1$, we have satisfied \eqref{i2} for $i-1$ in the case that $\tp(c_i/\bar c_{<i})$ is principal.

\subsubsection*{Case 2: $\tp(c_i/\bar c_{<i})$ is non-principal}
Condition \eqref{i1} for $i-1$ is easily satisfied, since we chose $D$ satisfying Lemma \ref{onlyicl}. Thus, we know that $F$ is continuous on $\cl(E^{i-1}_{\bar x})$. Note that the argument in the principal case for satisfying condition \eqref{i2} does not {\it a priori} work: there is no guarantee that $c_i<S(\bar c_{<i},z)$, as $\tp(c_i/\bar c_{<i})$ is non-principal, so the interval $(f_i(\bar c_{<i}),S(\bar c_{<i},z))$ might not contain $c_i$.

For $\bar x,\bar x'$ with $\bar x_{<i}=\bar x'_{<i}$, if we can bound $|\mu(\bar x_{\le i})-\mu(\bar x'_{\le i})|$ by some good bound at $i-1$, we will be done by the triangle inequality.  We may restrict $D$ so that $\mu$ is monotonic in the $i$-th coordinate.  If $\mu$ is constant in the $i$-th coordinate, then we have certainly bounded $\mu$ as desired, so we may suppose not.

Let $N=\Pr(\bar c_{<Q(i)})$.  Now consider $\mu^{-1}_{\bar c_{<i}}$.  Since $\tp(c_i/\bar c_{<i})$ is out of scale on $N$, this implies that $\mu^{-1}_{\bar c_{<i}}(N)$ is neither cofinal nor coinitial at $c_i$ in $\Pr(\bar c_{<i})$.  We can thus replace $f_i$ and $g_i$ by $\emptyset$-definable functions such that, for $y_i\in [f_i(\bar c_{<i}),g_i(\bar c_{<i})]$, we have $\mu(\bar c_{<i},y_i)\notin N$, and thus $\tp(\mu(\bar
c_{<i},y_i)/N)=\tp(\mu(\bar c_{<i},y'_i)/N)$ for any $y_i,y'_i\in[f_i(\bar
c_{<i}),g_i(\bar c_{<i})]$, since for two elements to have different types over
$N$, there must be an element of $N$ between them.

\begin{claim}
If $b,b'\in [f_i(\bar c_{<i}),g_i(\bar c_{<i})]$, then $\tp(|\mu(\bar c_{<i},b)-\mu(\bar c_{<i},b')|/N)$ is principal above $0$.
\end{claim}
\begin{proof}
Note that, since $\mu$ is a bounded function (since $F$ is), it cannot be the
case that $\mu(\bar c_{<i},b)$ is principal near $\pm\infty$ over $N$.  By Lemma
\ref{outofscaledfbl}, since $\tp(c_j/Nc_{<j})$ is principal, algebraic, or out of scale on $P$ for every $j\in[Q(i),i]$, we have that $\tp(\bar c_{\le i}/N)$ is definable, and hence $N\langle\bar c_{\le i}\rangle$ realizes only principal types over $N$, so $\tp(\mu(\bar c_{<i},b)/N)$ is principal.  Then $\tp(|\mu(\bar c_{<i},b)-\mu(\bar c_{<i},b')|/N)$ is principal near $0$, since two elements in the same finite principal type are separated by an infinitesimal amount, relative to $N$.
\end{proof}

Thus, the type of
\begin{equation*}
\tilde\mu(\bar c_{<i})=\sup\{|\mu(\bar c_{<i},x_i)-\mu(\bar c_{<i},x'_i)| :
x_i,x'_i\in [f_i(\bar c_{<i}),g_i(\bar c_{<i})]\}
\end{equation*}
over $N$ is principal near $0$.  Note that $\tilde\mu$ is $\emptyset$-definable as a function of $\bar c_{<i}$.

By induction (on $n$), we know that $\tilde\mu$ is continuous on the closure of some $\emptyset$-definable set containing $\bar c_{<i}$.  Since $\tilde\mu(\bar c_{<i})$ is principal near $0$ over $N$, the function $\tilde\mu$ must extend to $\icl(i-1,\bar c)$ as $0$.  Thus we may restrict $D$ and suppose that for all $\bar x\in D$, we have $\tilde\mu(\icl(i-1,\bar x))=0$.  Thus, $\tilde\mu$ is a good bound at $i-1$, by definition.  Let $G'$ be the good bound at $i-1$ bounding $G$ guaranteed by Lemma \ref{goodbounddown}.  Since $\tilde\mu(\bar x_{<i})\ge|\mu(\bar x_{\le i})-\mu(\bar x'_{\le i})|$ when $\bar x_{<i}=\bar x'_{<i}$, we can now satisfy \eqref{i2} for $i-1$: given $\bar a,\bar a'$ with $\bar a_{<i}=\bar a'_{<i}$,
\begin{multline*}
|F(\bar a)-F(\bar a')|\le |F(\bar a)-\mu(\bar a_{\le i})|+|F(\bar a')-\mu(\bar a'_{\le
  i})|+|\mu(\bar a_{\le i})-\mu(\bar a'_{\le i})|\le\\ 2G'(\bar
a_{<i})+\tilde\mu(\bar a_{<i}),
\end{multline*}
and thus we are done.

We now prove that failure of \eqref{theoremitemtype} implies failure of \eqref{theoremitemfunction}, so fix $p$ a finite decreasing $n$-type over $A$ not satisfying \eqref{theoremitemtype} and $\bar c\models p$. Once again, we suppose that $p$ satisfies Condition \ref{uniquecond} for all $i$ such that $Q(i)>0$. Fix $i$ the first coordinate such that $\tp(\bar c_{\le i}/A)$ does not satisfy \eqref{theoremitemtype}.

As before, we may suppose $A=\emptyset$.  We will construct a $\emptyset$-definable $i$-ary function, extending it to be constant on the last $n-i$ coordinates, so we may suppose that $i=n$.  Let $k=Q(n)$. Note that $k>0$, since else $p$ satisfies \eqref{theoremitemtype} with $i_0=n$. Let $N=\Pr(\bar c_{<k})$.  By hypothesis, there is some $\bar c_{<n}$-definable function, $f_{\bar c_{<n}}$, such that $f_{\bar c_{<n}}(N)$ is cofinal or coinitial at $c_n$ in $\Pr(\bar c_{<n})$.  Without loss of generality, suppose it is coinitial.  We may suppose that $f$ is monotonic by restricting its domain, and actually suppose that its domain is a finite interval after applying a definable homeomorphism from $(0,1)$ to $N$.  Define $F(x_1,\ldots,x_n)=f^{-1}_{\bar x_{<n}}(x_n)$. Note that $F$ is bounded. Let $C$ be any $\emptyset$-definable set containing $\bar c$ on which $F$ is continuous.  Using Lemma \ref{uniqueclosure}, we replace $C$ by a $\emptyset$-definable subset such that $\cl(C)$ contains exactly one point with first $k$ coordinates $\langle \bar c_{<k},\alpha(\bar c_{<k})\rangle$, where $\alpha$ is the $\emptyset$-definable function above which $c_k$ is principal.  We may further suppose that $C$ is a cell.  Let $g_n$ be the function bounding the $n$-th coordinate of $C$ from above.  Since $f_{\bar c_{<n}}(N)$ is coinitial at $c_n$ in $\Pr(\bar c_{<n})$, there is some element, $r\in N$, such that $c_n<f_{\bar c_{<n}}(r)<g_n(\bar c_{<n})$.  By coinitiality, we can then find $r'\in N$ with $c_n<f_{\bar c_{<n}}(r')<f_{\bar c_{<n}}(r)$.

Since $F(\bar c_{<n},f_{\bar c_{<n}}(r))=r$, and $F(\bar c_{<n},f_{\bar c_{<n}}(r'))=r'$, we must have non-empty $\bar c_{<k}$-definable sets $D_1=\{\bar x\in C\scht F(\bar x)=r\}$ and $D_2=\{\bar x\in C\scht F(\bar x)=r'\}$. Note that $D_1,D_2$ are each non-open in their last coordinate.

Again by Lemma \ref{uniqueclosure}, applied to $\tp(c_k,\ldots,c_{n-1}/N)$ and each of $\pi_{<n}(D_1)$ and $\pi_{<n}(D_2)$, we may shrink $D_1$ and $D_2$, keeping $\bar c_{<n}\in\pi_{<n}(D_1)\cap\pi_{<n}(D_2)$, and then suppose that there is a unique point in each of $\cl(D_1)$, $\cl(D_2)$ with first $k$ coordinates $\langle \bar c_{<k},\alpha_k(\bar c_{<k})\rangle$. But since both $\cl(D_1)$ and $\cl(D_2)$ are subsets of $\cl(C)$, and $\cl(C)$ has a unique such point, there is a common point in $\cl(D_1)$ and $\cl(D_2)$.  Since $F=r$ on $D_1$, and $F=r'$ on $D_2$, $F$ cannot be extended continuously to this common point, $\icl(k,\bar c)$.  Now observe that $D_1,D_2,r,r'$ can all be regarded as parametrized sets and functions of $\bar c_{<k}$, and so $F$ cannot be extended continuously to $\icl(k,\bar x)$ for $\bar x$ in some open set containing $\bar c$.

Having established the Theorem \ref{maintheoremrestate} for decreasing independent types $p$ satisfying Condition \ref{uniquecond} for all $i$ with $Q(i)>0$, we show how to reduce the other cases to this one.

\begin{claim}\label{homeomorphism-prop}
Let $p$ and $q$ be types over $\emptyset$, contained in closed sets $B'$ and $B$, respectively, such that $f$ is a $\emptyset$-definable homeomorphism from $B'$ to $B$ with $f(p)=q$. Then \eqref{theoremitemfunction} holds for $p$ if and only if it holds for $q$.
\end{claim}
\begin{proof}
Since the situation is symmetric, we suppose that \eqref{theoremitemfunction} holds for $p$ and prove it holds for $q$. Let $\bar c\models p$. Let $F$ be any $\emptyset$-definable bounded function defined on $q$. Then $F\circ f$ is a $\emptyset$-definable function defined on $\bar c$. Applying \eqref{theoremitemfunction}, we can find a $\emptyset$-definable set $C$ containing $\bar c$ such that $F\circ f\restrict C$ is continuous and extends continuously to $\cl(C)$. Let $C'=f(C)\cap B$. Then $C'$ is a $\emptyset$-definable set containing $q$ such that $F\restrict C'$ is continuous. Moreover, since $f^{-1}(B)$ is defined, and $\cl(C')\subseteq B$, we know that $F\circ f$ extends continuously to $f^{-1}(\cl(C'))$, so $F\restrict C'$ extends continuously to $\cl(C')$, showing that \eqref{theoremitemfunction} holds for $q$.
\end{proof}

\begin{claim}\label{claimindependent}
Let $p$ be a decreasing $n$-type. There is $p'$ a decreasing independent type
satisfying each one of \eqref{theoremitemfunction} and \eqref{theoremitemtype}
if and only if $p$ does, and finite if $p$ is.
\end{claim}
\begin{proof}
Suppose that $p$ is not independent. Let $D$ be a closed $\emptyset$-definable set of lowest dimension containing $p$. Then the projection map $p_D$ defined in \cite{vdD98book}, Chapter 3, 2.7 is a homeomorphism from $D$ into $\mathcal C^{\dim(D)}$, and $D$ and $p_D(D)$ are the desired $B',B$ in Claim \ref{homeomorphism-prop}, so Claim \ref{homeomorphism-prop} gives equivalence of satisfaction of \eqref{theoremitemfunction}, and it is easy to see that $p_D$ preserves satisfaction of \eqref{theoremitemtype}.
\end{proof}

\begin{claim}\label{claimnoncutabove}
Let $p$ be a finite decreasing independent $n$-type. There is $p'$ a finite decreasing independent $n$-type such that $p'$ satisfies Condition \ref{uniquecond} for all $i\le n$ with $Q(i)>0$, and $p'$ satisfies each one of \eqref{theoremitemfunction} and \eqref{theoremitemtype} if and only if $p$ does.
\end{claim}
\begin{proof}
Let $\bar c\models p$. We modify $\bar c$ in stages.  At stage $i$, we suppose by induction that for each $k<i$ and $j\ge Q(i)$, the type $\tp(c_j/\bar c_{<Q(k)})$ is principal above some element of $\dcl(c_{<Q(k)})$.  Suppose
that for some $j\ge Q(i)$, the element $c_j$ is not principal above an element
of $\dcl(c_{<Q(i)})$.  By Lemma \ref{noncutsafterni}, $c_j$ is principal near some $\beta\in\dcl(c_{<Q(i)})$, so it is principal below $\beta$.  Let
$c'_j=\beta+(\beta-c_j)$, so $\tp(c'_j/\bar c_{<Q(i)})$ is principal above $\beta$.  For all $k<i$, we know that $c_j$ is principal above
an element of $\dcl(c_{<Q(k)})$, say $\beta'$.  Then $\beta$ is also principal
above $\beta'$, and so $c'_j$ is also principal above $\beta'$.  Hence replacing $c_j$ by $c'_j$ preserves the fact that $c_j$ is principal above an element of $\dcl(c_{<Q(k)})$ for all $k<i$.  We do this for each such $j$.  Then after $n$ stages, $\tp(\bar c)$ satisfies Condition \ref{uniquecond} for all $i\le n$ such that $Q(i)>0$.  It is easy to verify that the inverse of the composition of the functions applied to $\bar c$ in this process satisfies the conditions of Claim \ref{homeomorphism-prop}, and also preserves satisfaction of \eqref{theoremitemtype}.
\end{proof}

With Claims \ref{claimindependent} and \ref{claimnoncutabove}, we have shown that we lost no generality in restricting to considering independent types that satisfy Condition \ref{uniquecond} for $i$ such that $Q(i)>0$.
\long\def\janakqed{\textsquare}
\renewcommand\qedsymbol{{\it (Theorem \ref{maintheoremrestate})} \janakqed}
\end{proof}

In the general case, when $p$ is not a finite type, we need not restrict ourselves solely to bounded functions. Recall that a (possibly nonlinear) operator on topological vector spaces is called \emph{bounded} if the image of a bounded set is bounded. Note that if $F$ is bounded as a function, $F$ is {\it a fortiori} bounded as an operator.

\begin{corollary}\label{infinitetype}
Let $M$ be an o-minimal field, and let $A\subseteq M$.  Let $p$ be a decreasing $n$-type over $A$.  Let $\bar c=\langle c_1,\ldots,c_n\rangle\models p$.  Then the following statements are equivalent:
\begin{enumerate}
\item\label{coritemfunction} For every bounded (as an operator) $A$-definable $n$-ary function, $F$, defined on $\bar c$, there is an $A$-definable set $C$ with $\bar c\in C$, such that $F\restrict C$ is continuous and extends continuously to $\cl(C)$.
\item\label{coritemtype} There is $i_0\le n$ such that $\tp(c_i/A\bar c_{<i})$ is algebraic, principal, or out of scale on $A\bar c_{<Q(i)}$ for $i=i_0,\ldots,n$, and either $\tp(c_{i_0-1}/A)$ is non-principal, or for some $j\in[Q(i_0-1),i_0-1]$, we have $\tp(c_j/A\bar c_{Q(i_0-1)})$ principal near $\pm\infty$.
\end{enumerate}
\end{corollary}
\begin{proof}
If $p$ is finite, Corollary \ref{infinitetype} reduces to Theorem \ref{maintheoremrestate}. Thus, we may suppose that $p$ is not contained in any bounded definable set. If $p$ does not satisfy \eqref{coritemtype}, then it is easy to see that the proof that failure of \eqref{theoremitemtype} implies failure of \eqref{theoremitemfunction} in Theorem \ref{maintheoremrestate} works verbatim. It only remains to show that if $p$ is a decreasing $n$-type over $A$ satisfying \eqref{coritemtype} with $\tp(c_j/A\bar c_{Q(i_0-1)})$ principal near $\pm\infty$ for some $j\in[Q(i_0-1),i_0-1]$, then $p$ satisfies \eqref{coritemfunction}. Claim \ref{claimindependent} shows we can take $p$ to be independent, and Claim \ref{claimnoncutabove} can be suitably modified so that $p$ satisfies Condition \ref{uniquecond} for all $i\ge i_0$ such that $Q(i)>0$. Let $F$ be a bounded (as an operator) $A$-definable function defined on $\bar c$. By Lemma \ref{onlyicl}, given any $A$-definable set $C$ containing $\bar c$ with $F\restrict C$ continuous, we may take $C'\subseteq C$ such that $\cl(C')\setminus C\subseteq\bigcup_{i\le n}\icl(i,C')$. Note that if, for some $i$ and $j\ge Q(i)$, the type $\tp(c_j/A\bar c_{<Q(i)})$ is principal near $\pm\infty$, then $\icl(i,C')$ is not defined, so empty. Then the proof of Theorem \ref{maintheoremrestate} proceeds as before, until the coordinate $i_0-1$. Note that in the proof, we take limits of $F$ only in finite neighborhoods and suprema only along bounded coordinates, so these limits and suprema are still defined for an $F$ that is bounded as an operator. At stage $i_0-1$, $F$ has been continuously extended onto every point in $\cl(C')$, and so the proof finishes there, possibly after some further applications of Corollary 2.4 of Chapter 6 of \cite{vdD98book} to ensure continuity across fibers.
\end{proof}

\bibliography{../janak}
\bibliographystyle{alpha-elink}

\end{document}